\RequirePackage{fix-cm}
\documentclass[smallextended]{svjour3}       
\smartqed  
\usepackage{
amsmath,
amssymb,
hyperref,
mathrsfs,
graphicx,
stmaryrd,
amscd,
 amsfonts, stmaryrd,latexsym, xargs}
\usepackage{color}

\usepackage[colorinlistoftodos]{todonotes}
 \presetkeys{todonotes}%
{inline,backgroundcolor=gray!20,bordercolor=gray!30}{}
\tikzset{/tikz/notestyleraw/.append style={text=black}}

\newtheorem{thm}{Theorem}[section]

\newtheorem{lem}[thm]{Lemma}
\newtheorem{defn}[thm]{Definition}

\newtheorem{prop}[thm]{Proposition}

\newtheorem{rmk}[thm]{Remark}
\newcommand{\be}{\begin{eqnarray}}
\newcommand{\ee}{\end{eqnarray}}
\newcommand{\beq}{\begin{equation}}
\newcommand{\eeq}{\end{equation}}
\newcommand{\ben}{\begin{eqnarray*}}
\newcommand{\een}{\end{eqnarray*}}

\newcommand{\beal}{\begin{aligned}}
\newcommand{\enal}{\end{aligned}}
\newcommand{\bc}{\begin{cases}}
\newcommand{\ec}{\end{cases}}

\newcommand{\R}{\mathbb{R}}

\newcommand{\PP}{\mathbb{P}}

\newcommand{\cF}{\mathcal{F}}

\newcommand{\cA}{\mathcal{A}}

\newcommand{\cM}{\mathcal{M}}

\newcommand{\wt}{\widetilde}
\newcommand{\wh}{\widehat}

\newcommand{\om}{\omega}
\newcommand{\lb}{\lambda}

\newcommand{\T}{\mathbb{T}}

	\def\textr{\textcolor{red}}

\begin{document}

\title{Essential forward weak KAM solution for the convex Hamilton-Jacobi equation
}


\author{
Xifeng Su      
\and Jianlu Zhang
}


\institute{Xifeng Su\at Laboratory of Mathematics and Complex Systems (Ministry of Education), School of Mathematical Sciences, Beijing Normal University, Beijing 100875, People's Republic of China
\\
Tel: +86-15910594182\\
\email{xfsu@bnu.edu.cn, billy3492@gmail.com}
\and
Jianlu Zhang \at
              Hua Loo-Keng Key Laboratory of Mathematics \& Mathematics Institute, Academy of Mathematics and systems science, Chinese Academy of Sciences, Beijing 100190, China \\
              Tel: +86-182-1038-3625\\
              \email{jellychung1987@gmail.com}           
}

\date{Received: date / Accepted: date}

\maketitle

\begin{abstract}
For a convex, coercive continuous Hamiltonian on a compact closed Riemannian manifold $M$, we construct a unique forward weak KAM solution of 
\[
H(x, d_x u)=c(H)
\]
by a vanishing discount approach, where $c(H)$ is the Ma\~n\'e critical value. We also discuss the dynamical significance of such a special solution.
\end{abstract}


\section{Introduction}\label{s1}

For a compact connected manifold $M$ without boundary, the {\sf Hamiltonian} is usually mentioned as a continuous function  defined on its cotangent bundle $T^*M$. In \cite{LPV}, the authors firstly proposed the {\sf ergodic approximation} technique, to consider the existence of so called {\sf viscosity solutions} to the {\sf Hamilton-Jacobi equation}
\beq\label{eq:hj}\tag{HJ$_0$}
H(x, d_x u)=c(H),\quad x\in M
\eeq
for the {\sf Ma\~n\'e critical value} 
\[
c(H):=\inf\{c \in\R | \exists\ \om\in C(M,\R) \text{ such that }H(x,d_x\om)\leq c,\quad a.e.\ x\in M\}.
\]
The Hamiltonian they concerned satisfies
\begin{itemize}
\item  {\sf (Coercivity)} $H(x,p)$ is coercive in $p\in T_x^*M$, uniformly w.r.t. $x\in M$.
\item {\sf (Convexity)} $H(x,p)$ is convex in $p\in T_x^*M$ for all $x\in M$.
\end{itemize}
 They perturbed (\ref{eq:hj}) by the following {\sf discounted equation}
\be\label{eq:d-hj}
\lb u+H(x, d_x u)=c(H),\quad x\in M, \lb >0
\ee
of which the {\sf Comparison Principle} is allowed. Therefore, the viscosity solution $u_\lb^-$ of (\ref{eq:d-hj}) is unique.  In \cite{DFIZ}, they established the convergence of $u_\lb^-$ as $\lb\rightarrow 0_+$, to a specified viscosity solution $u_0^-$ of (\ref{eq:hj}) which can be characterized by the combination of {\sf subsolutions} of (\ref{eq:hj}), 
or {\sf Peierls barrier} 
\[
h^\infty:M\times M\rightarrow \R
\]
w.r.t. the {\sf projected Mather measures $\mathfrak M$} of (\ref{eq:hj}), see Appendix \ref{a1} for the relevant definitions of $\mathfrak M,\ h^\infty$, subsolutions etc.\\

In this paper, we consider a negative limit technique and try to find another specified solution of (\ref{eq:hj}). Precisely, we consider 
\beq\label{eq:d-hj-n}\tag{HJ$_\lb$}
-\lb u+H(x,d_x u)=c(H),\quad x\in M, \lb>0
\eeq
of which a unique {\sf forward $\lb-$weak KAM solution} $u_\lb^+$ can be found (see Remark~\ref{defn:lb-w-kam} for the definition). As $\lb\rightarrow 0_+$, we get the following conclusion:

\begin{thm}\label{thm:1}
Let $H : T^*M \rightarrow\R, (x, p) \mapsto H(x,p)$ be a continuous Hamiltonian  coercive and convex in $p$. For $\lb>0$, the unique forward $\lb-$weak KAM solution $u_\lb^+$ of (\ref{eq:d-hj-n}) uniformly converges as $\lb\rightarrow 0_+$, to a unique {\sf forward $0-$weak KAM solution}\footnote{see Definition \ref{defn:w-kam}} $u_0^+$ of (\ref{eq:hj}), which can be interpreted as 
\be\label{eq:1}
u_0^+(x)=\inf\cF_+
\ee
with
\be
\cF_+:=\Big\{w\text{ is a subsolution of (\ref{eq:hj}) }\Big|\int_Mw d\mu\geq 0,\  \forall \mu\in \mathfrak M\Big\}
\ee
and
\be\label{eq:2}
 u_0^+(x)=-\inf_{\mu\in\mathfrak M}\int_Mh^\infty(x,y)d\mu(y).
\ee
\end{thm}

\begin{rmk}
The novelty of this paper is that we adapt a symmetric Lagrangian skill to our $C^0-$setting. The lack of regularity invalidates a bunch of important properties of the Mather measures, Peierls barrier etc., so we have to find substitutes in the weak sense.

Besides,  we  mention that $-u_0^+(x)$ is a viscosity solution of the symmetric equation
\[
H(x,-d_x u(x))=c(H).
\]
Comparing to the backward $0-$weak KAM solutions, the notion of viscosity solutions is more familiar to PDE specialists, although both are proved to be equivalent in \cite{F}.
%
\end{rmk}



\subsection{Dynamic interpretation of $u_0^+$}

Now the vanishing discount approach supplies us with a pair of solutions of \eqref{eq:hj}:
\beq
\left\{
\beal
&u_0^-(x)=\inf_{\mu\in\mathfrak M}\int_Mh^\infty(y,x)d\mu(y),\\
&u_0^+(x)=-\inf_{\mu\in\mathfrak M}\int_Mh^\infty(x,y)d\mu(y).
\enal
\right.
\eeq
\begin{defn}[Conjugated Pair]
A backward $0-$weak KAM solution $u^-$ of \eqref{eq:hj} is conjugated to a forward  $0-$weak KAM solution $u^+$, if 
\begin{itemize}
\item  $u^-=u^+$ on the {\sf projected Mather set} $\cM$ (see Definition \ref{defn:mat}). 
\item  $u^-\geq u^+$ on $M$.
\end{itemize}
\end{defn}
\textr{In \cite{F}, above definition is actually proposed to $C^2-$Tonelli Hamiltonian\footnote{$H:(x,p)\in T^*M\rightarrow\R$ is called {\sf Tonelli}, if it's positive definite and superlinear in $p$}. However, we have no difficulty to reserve this concept to the $C^0-$case by means of the evidence in \cite{DS}.} Moreover, For the following typical Hamiltonians, $(u_0^-, u_0^+)$ indeed forms a {\sf conjugated pair}:
\begin{enumerate}
\item {\bf (uniquely ergodic)} Suppose $\mathfrak M$ consists of a uniquely ergodic projected Mather measure \textr{(generic for $C^2-$Tonelli Hamiltonians, see \cite{Mn})}, then 
\[
d_c(x,y):=h^\infty(x,y)+h^\infty(y,x)\geq 0
\]
for all $x,y\in M$, and `$=$' holds for $x,y\in \cM$ (due to the definition of the Mather measure in Appendix \ref{a1}
). So $(u_0^-, u_0^+)$ is a conjugated pair.\\
\item {\bf (mechanical system)} For a mechanical Hamiltonian
\[
H(x,p)=\frac12 \langle p,p\rangle+V(x), 
\]
we can easily get $c(H)=\max_{x\in M}V(x)$, then the associated $L(x,v)+c(H)\geq 0$ on $TM$. Consequently, $h^\infty:M\times M\rightarrow\R$ is nonnegative, so $u_0^-\geq u_0^+$ on $M$. On the other side, due to the definition of $\mathfrak M$, all the Mather measures are supported by equilibriums. So $u_0^-=u_0^+$ on $\cM$. In summary $(u_0^-, u_0^+)$ is a conjugated pair.\\
\item {\bf (constant subsolution)} Such a case is also discussed in \cite{DFIZ}. If $H(x,0)\leq c(H)$ for all $x\in M$, i.e. constant is a subsolution of \eqref{eq:hj}, then due to the {\sf Young Inequality} we get 
\[
L(x,v)\footnote{the Lagrangian function is defined in \eqref{eq:leg}}
+c(H)\geq L(x,v)+H(x,0)\geq \langle v,0\rangle=0
\]
for all $(x,v)\in TM$, by a similar analysis like the case of mechanical systems $(u_0^-, u_0^+)$ proves to be a conjugated pair. 
\end{enumerate}

\subsection{Organization of the article.}\label{s1.6}

 In Sec. \ref{s2}, we prove some variational properties for nonsmooth Lagrangians. In Sec. \ref{s3}, we prove the convergence of $u_\lb^+$ as $\lb\rightarrow 0_+$ and give a representative formula for the limit. 
 For the consistency and readability of the article, some preliminary materials are moved to Appendix.
\vspace{10pt}

\noindent{\bf Acknowledgement.} J. Zhang is supported by the National Natural Science Foundation of China (Grant No. 11901560). X. Su is supported by the National Natural Science Foundation of China (Grant No. 11971060, 11871242).

\vspace{20pt}

\section{Nonsmooth symmetric Lagrangians}\label{s2}

With the same adaption as in \cite{DFIZ}, without loss of generality we can assume $H(x,p)$ is superlinear in $p$, i.e.
\begin{itemize}
\item {\sf (Superlinearity)} $\lim_{|p|\rightarrow +\infty}H(x,p)/|p|=+\infty$, for any $x\in M$.
\end{itemize}
 In that case, by Fenchel's formula (see \ref{eq:leg}), the Hamiltonian has an associated {\sf Lagrangian} $L :(x,v)\in  T M \rightarrow \R$ which is superlinear and convex in the fibers of the tangent bundle. 
Consequently,  we  can propose a {\sf symmetrical Lagrangian}  $\wh L(x,v):=L(x,-v)$, of which the following fundamental facts hold:
\begin{lem}\label{lem:sym-h}
\begin{itemize}
\item [(i)] The conjugated Hamiltonian $\wh H: T^*M\rightarrow\R$ of $\wh L(x,v)$ satisfies 
$\wh H(x,p)=H(x,-p)$ for all $(x,p)\in T^*M$. Therefore, $\wh H$ is also continuous, superlinear and convex. 
\item [(ii)] $\wh H(x, d_x\om)\leq c \iff  H(x,d_x(-\om))\leq c$
\item [(iii)] $c(\wh H) = c(H)$.
\item [(iv)] The projected Mather measure set $\wh{\mathfrak M}$ (associated with $\wh H(x,p)$) 
keeps the same with $\mathfrak M$.
\item [(v)] The Peierls barrier function associated with $\wh L(x,v)$ satisfies $\wh h^\infty(y,x)= h^\infty(x,y)$ for any $x,y\in M$.
\end{itemize}
\end{lem}

\begin{proof}
(i). Due to \eqref{eq:leg} and the definition of $\wh L$, we have
\[
\begin{split}
\wh H (x,p) &= \max_{v\in T_x M} \{\langle p,v\rangle - \wh L (x,v) \}\\
 & = \max_{v\in T_x M} \{\langle p,v\rangle -  L (x, -v) \}\\
 & =  \max_{w\in T_x M} \{\langle -p,w\rangle -  L (x, w) \} = H(x, -p),
\end{split}. 
\]
as is desired.

(ii). If $\om$ is a subsolution of $\wh H (x,d_x\om)\leq c$, then for any absolutely continuous $\gamma:[-T,T]\rightarrow M$ connecting $x,y\in M$, we get
\[
\om(y)-\om(x)\leq \int_{-T}^T\big(\wh L(\gamma,\dot\gamma)+c\big) dt.
\]
Furthermore, 
\ben
-\om(x)-(-\om(y))&=&\om(y)-\om(x)\\
&\leq &\int_{-T}^T\big( \wh L(\gamma,\dot\gamma)+c \big) dt\\
&=& \int_{-T}^T\big( L(\gamma,-\dot\gamma)+c\big) dt\\
&=& \int_{-T}^T \big( L(\wh \gamma,-\dot{\wh\gamma})+c\big) dt
\een
with $\wh\gamma(t):=\gamma(-t)$ for all $t\in[-T,T]$. As $\gamma$ is arbitrarily chosen, so we get $-\om\prec L+c$\footnote{see Definition \ref{defn:sub-sol}}, then $H(x,-d_x\om)\leq c$ for a.e. $x\in M$. Similarly, $\om\prec L+c$ indicates $-\om\prec\wh L+c$.

 (iii).  As $c(\wh H)=\inf\{c \in\R |  \exists\ \om\in C(M,\R) \text{ such that }\om\prec \wh L+c\}$, then due to (ii), $c(\wh H)=c(H)$.

(iv). Due to \textr{Proposition 2-4.3} of \cite{CI}, for any measure $\wt\mu\in\wt{\mathfrak M}$, there exists a sequence of closed measures $\wt\mu_n\in \PP_c (TM)$ (defined in Appendix \ref{a1}), such that $\wt\mu_n$ weakly converges to $\wt\mu$ and 
\[
\lim_{n\rightarrow +\infty}\int_{TM}Ld\wt\mu_n=\int_{TM}Ld\wt \mu.
\]
Moreover, for each $\wt \mu_n$ there exists an absolutely continuous curve $\gamma_n:t\in[-T_n,T_n]\rightarrow M$ with $T_n\rightarrow +\infty$ as $n\rightarrow+\infty$, such that 
\[
\int_{TM} gd\mu_n=\frac{1}{2T_n}\int_{-T_n}^{T_n}g(\gamma_n(t),\dot\gamma_n(t)) dt,\quad\forall\\ g\in C_c(TM,\R).
\]
Therefore, for  $\wh\gamma_n(t):=\gamma_n(-t)$,
 we have 
\ben
-c(H)&=&\lim_{n\rightarrow+\infty}\frac 1{2T_n}\int_{-T_n}^{T_n}L(\gamma_n(t),\dot\gamma_n(t))dt\\
&=&\lim_{n\rightarrow+\infty}\frac 1{2T_n}\int_{-T_n}^{T_n}L(\wh \gamma_n(-t),-\dot{\wh\gamma}_n(-t))dt\\
&=&\lim_{n\rightarrow+\infty}\frac 1{2T_n}\int_{-T_n}^{T_n}\wh L(\wh \gamma_n(-t),\dot{\wh\gamma}_n(-t))dt\\
&=&\lim_{n\rightarrow+\infty}\frac 1{2T_n}\int_{-T_n}^{T_n}\wh L(\wh \gamma_n(s),\dot{\wh\gamma}_n(s))ds\\&=&-c(\wh H).
\een
That indicates $S^*\wt\mu$ is a Mather measure for $\wh L(x,v)$, where $S:TM\rightarrow TM$ is a diffeomorphism defined by $S(x,v)=(x,-v)$. Namely, we have
\[
\int_{TM}g(x,v)d S^*\wt\mu(x,v):=\int_{TM} g(x,-v) d\wt\mu(x,v),\quad\forall g\in C_c(TM,\R).
\]
Due to (\ref{eq:pro-mes}) and $S\circ S=id$, 
\ben
\int_Mf(x)d\mu(x)&=&\int_{TM}f\circ\pi(x,v)d\wt\mu(x,v)\\
&=&\int_{TM}f\circ\pi(x,-v)dS^*\wt\mu(x,v)\\
&=&\int_M f(x) d\pi^*S^*\wt\mu(x,v)
\een
for all $f\in C(M,\R)$, then $\pi^*S^*\wt\mu=\mu\in\mathfrak M$. So $\wh{\mathfrak M}=\mathfrak M$.

(v). Due to the definition of the Peierls barrier function, we calculate
\ben
\wh h^\infty(y,x) &=& \liminf_{t\rightarrow +\infty} \left( \inf_{\substack{\xi\in C^{ac}([0,t],M)\\
\xi(0)=y\\
\xi(t)=x}}\int_0^{t} \wh L(\xi(s),\dot{\xi}(s)) ds+c(H)t \right) \\
&=&\liminf_{t\rightarrow +\infty} \left( \inf_{\substack{\gamma\in C^{ac}([0,t],M)\\
\gamma(t)=y\\
\gamma(0)=x}}\int_0^{t}  L\big(\gamma(t-s),-\frac{d\gamma(t-s)}{ds}\big) ds+c(H)t \right)\\
&=&  \liminf_{t\rightarrow +\infty} \left( \inf_{\substack{\gamma\in C^{ac}([0,t],M)\\
\gamma(0)=x\\
\gamma(t)=y}}\int_0^{t}  L(\gamma(\tau),\dot{\gamma}(\tau)) d\tau+c(H)t \right)\\
& = &h^\infty(x,y),
\een
where $\xi(s) := \gamma(t-s)$ for $s\in [0,t]$.
\qed
\end{proof}


Now we can propose a function
\be\label{eq:sta}
u_\lb ^+(x):=-\inf_{\substack{\gamma\in C^{ac}([0,+\infty),M)\\\gamma(0)=x}}\int_0^{+\infty}e^{-\lb t} \big(L(\gamma(t),\dot\gamma(t)) + c(H) \big) \ dt,
\ee
of which the following properties hold:

\begin{prop}\label{prop:dom-fun}
\begin{enumerate}
\item [(1)]For $\lb\in(0,1]$, $u_\lb^+$ is uniformly Lipschitz and equi-bounded on $M$, with the Lipschitz (resp. equi-bound) constant depending only  on $L$.

\item [(2)]For every $x\in M$, there {exists} a forward curve $\gamma_{\lb,x}^+:[0,+\infty)\rightarrow M$ which achieves the minimum of (\ref{eq:sta}). 

\item [(3)]\textsf{(Domination)} $u_\lb^+$ is $\lambda$-dominated by $L$ and is denoted by  $u_\lb^+\prec_{-\lb}L+c(H)$, i.e., for any $(x,y)\in M\times M$ and $b-a\geq 0$, we have 
\[
e^{-\lb b}u_\lb^+(y)-e^{-\lb a}u_\lb^+(x)\leq \inf_{\substack{\gamma\in C^{ac}([a,b],M)\\\gamma(a)=x,\gamma(b)=y}}\int_a^be^{-\lb s}\Big(L(\gamma,\dot\gamma)+c(H)\Big)ds
\]

\item  [(4)]\textsf{(Calibration)} For any $t\geq0$,
\[
u_\lb^+(x)=e^{-\lb t}u_\lb^+(\gamma_{\lb,x}^+(t))+\int_t^0e^{-\lb s}\Big(L(\gamma_{\lb,x}^+(s),\dot\gamma_{\lb,x}^+(s))+c(H)\Big)ds.
\]
 Such a curve $\gamma_{\lambda, x}^+$  is called a forward calibrated curve of $u_\lb^+$.
  
  \item [(5)] $-u_\lb^+(x)$ is the viscosity solution of the following symmetrical H-J equation 
\be\label{eq:hj-sym}
\lb u+\wh H(x,\partial_x u)=c(H),\quad\lb>0.
\ee

\end{enumerate}
\end{prop}

\begin{rmk}[forward $\lb-$weak KAM solution]\label{defn:lb-w-kam}
A continuous function $w:M\rightarrow\R$ is called a {\sf forward $\lb-$weak KAM solution} of (\ref{eq:d-hj-n}) if it satisfies item (3) and (4) of Proposition \ref{prop:dom-fun}. Due to Property (5) of Proposition \ref{prop:dom-fun}, such a forward $\lb-$weak KAM solution is unique.
\end{rmk}
\vspace{10pt}

\noindent{\it Proof of Proposition \ref{prop:dom-fun}:} 
(5) By a simple transformation, we can see 
\ben
-u_\lb^+(x)=\inf_{\substack{\xi\in C^{ac}((-\infty,0],M)\\\xi(0)=x}}\int^0_{-\infty}e^{-\lb t} \big(\wh L(\xi(t),\dot\xi(t)) + c(\wh H) \big) \ dt.
\een
Due to Appendix 2 of \cite{DFIZ}, $-u_\lb^+$ is a viscosity solution of \eqref{eq:hj-sym}, which is unique due to the { Comparison Principle}. \medskip

(1) For any viscosity solution $u_0(x)$ of
\eqref{eq:hj}, we get 
\[
\underline{u}_0(x):=u_0(x)-\|u_0\|\leq 0\leq u_0(x)+\|u_0\|:=\bar u_0(x),\quad\forall x\in M.
\]
Consequently, $\underline u_0$ (resp. $\bar u_0$) is a subsolution (resp. supersolution) of \eqref{eq:hj-sym}. Due to the Comparison Principle again, for any $\lb>0$ and any viscosity solution $\om_\lb$ of \eqref{eq:hj-sym} satisfies $\underline u_0\leq \om_\lb\leq \bar u_0$.
So we get the equi-boundedness of $\{u_\lb^+\}_{\lb\in(0,1]}$.\medskip

Let $\gamma: [0, d(x,y)] \rightarrow M$ be the geodesic joining  $y$ to $x$ parameterized by the arc-length, \textr{where $d:M\times M\rightarrow\R$ is the Euclidean distance}. For every absolute continuous curve $\xi: [0, +\infty) \rightarrow M$ with $\xi(0) =x$, we define a curve 
\begin{equation}
\eta(t)=
\left\{
 \begin{aligned}
 &\gamma(t), \qquad\qquad\qquad t\in [0, d(x,y)],\\
 &\xi(t-d(x, y)),\qquad t\in [d(x,y), +\infty).
 \end{aligned}
 \right.
\end{equation}
Then we have
\[
\begin{split}
-u_\lambda^+(y)  & \leq \int_0^{+\infty}e^{-\lb t} \big(L(\eta(t),\dot\eta(t)) + c(H) \big) \ dt\\
&\leq \int_0^{d(x,y)}e^{-\lb t} \big(L(\gamma(t),\dot\gamma(t)) + c(H) \big) \ dt \\ &\quad+ \int_{d(x,y)}^{+\infty}e^{-\lb t} \big(L(\xi(t-d(x,y) ) ),\dot\xi(t-d(x,y) ) ) + c(H) \big) \ dt\\
&\leq  \int_0^{d(x,y)}e^{-\lb t} \big(L(\gamma(t),\dot\gamma(t)) + c(H) \big) \ dt \\
&\quad  + e^{\lambda d(x,y)} \int_{0}^{+\infty}e^{-\lb t} \big(L(\xi(t) ),\textr{\dot\xi(t)} ) + c(H) \big) \ dt.
\end{split}
\]
By minimizing with respect to all $\xi\in C^{ac}\big([0,+\infty)\big)$ with $\xi(0) = x$, we obtain
\[
-u_\lambda^+(y) \leq -e^{\lambda d(x,y)} u_\lambda^+(x) + \int_0^{d(x,y)}e^{-\lb t} \big(L(\gamma(t),\dot\gamma(t)) + c(H) \big) \ dt.
\]
Therefore, we have
\[
\begin{split}
u_\lambda^+(x) - u_\lambda^+(y) &\leq \big(1-e^{\lambda d(x,y)} \big) u_\lambda^+(x) + \int_0^{d(x,y)}e^{-\lb t} \big(L(\gamma(t),\dot\gamma(t)) + c(H) \big) \ dt \\
&\leq \frac{1-e^{\lambda d(x,y)} }{\lambda} \left( |\lambda u_\lambda^+(x)| + C_1\right) \leq (C+C_1) d(x,y),
\end{split}
\]where $C:=\max\{|\max_{x\in M}L(x,0)+c(H)|, |\min_{(x,v)\in TM}L(x,v)+ c(H)\}$ and $C_1:=\max\{ L(x,v)~:~ x\in M , \|v\|\leq 1\}$. 
By exchanging the role of $x$ and $y$, we get the other inequality, which shows that $u_\lambda^+$ is uniformly Lipschitz and the Lipschitz constant is independent of $\lb$. \medskip

(2), (3) \& (4) By a similar analysis as Propositions 6.2--6.3 in \cite{DFIZ}, all these three items can be easily proved.
\qed



\section{Discounted limit of forward $\lb-$weak KAM solutions}\label{s3}

Recall that $\wh u_\lb^-:=-u_\lb^+$  is the unique  viscosity solution of (\ref{eq:hj-sym}), 
\[
\wh u_\lb^-(x)=\inf_{\gamma(0)=x}\int_{-\infty}^0 e^{\lb  t}\big(\wh L(\gamma,\dot\gamma) + c(H)\big)dt.
\]
So the following conclusion holds instantly:

\begin{lem}\cite{DFIZ}\label{lem:sym}
As $\lb\rightarrow 0_+$, $\wh u_\lb^-$ converges to a unique function $\wh u_0^-$ which is a viscosity solution of the following
\be\label{eq:hj-sym0}
\wh H(x,\partial_x u)=c( H)
\ee
with the following two different interpretations:
\begin{itemize}
\item $\wh u_0^-$ is the maximal subsolution $w:M\rightarrow \R$ of \eqref{eq:hj-sym0}
 such that for any projected Mather measure $\wh \mu\in\wh{\mathfrak M}$, $\int_Mw\cdot d{\wh \mu}\leq 0$.
\item  $\wh u_0^-$ is the infimum of functions $\wh h_\mu^\infty$ defined by 
\[
\wh h_\mu^\infty(x):=\int_M \wh h^\infty(y,x) d\wh\mu(y),\quad \wh\mu\in\wh{\mathfrak M}.
\]
\end{itemize}
\end{lem}
Due to Lemma \ref{lem:sym-h} and Lemma \ref{lem:sym}, we get 
\[
u_0^+:=\lim_{\lb\rightarrow 0_+}u_\lb^+=-\lim_{\lb\rightarrow 0_+}\wh u_\lb^-=-\wh u_0^-
\]
which is uniquely established and interpreted as the following:

\begin{lem}\label{lem:for}
$u_0^+$ is a forward $0-$weak KAM solution
of (\ref{eq:hj}). 
\end{lem}
\begin{proof}
As $\wh u_0^-$ is the viscosity solution of (\ref{eq:hj-sym0}), then \textr{$\wh u_0^-\prec \wh L+c(H)$ due to Proposition 5.3 of \cite{DFIZ}. On the other side, due to 
$$
U(x,t):=\inf_{\substack{\gamma\in C^{ac}([0,t],M) \\\gamma(0)=x }}\{\wh u_0^-(\gamma(-t))+\int^0_{-t}\wh{L}(\gamma(\tau),\dot{\gamma}(\tau))+c(H)\mbox{d}\tau\}, \quad\forall t\geq 0, 
$$
is the unique viscosity solution of the Cauchy problem 
 \[
 \left\{
 \begin{aligned}
 &\partial_tu+\wh{H}(x,d_xu)=c(H)\\
 &u(x,0)=\wh u_0^-(x),\quad t\geq 0, 
 \end{aligned}
 \right.
 \]
whereas $\wh u_0^-(x)$ is also a viscosity solution of the Cauchy problem. So it follows that $U(x,t)=\wh u_0^-(x)$ for all $x\in M,t>0$. Hence, by the same analysis as in the proof of
Proposition 6.2 of \cite{DFIZ}, for any $x\in M$, there exists a curve $\gamma_x^-:(-\infty,0]\rightarrow M$ absolutely continuous and ending with $x$, such that 
\[
\wh u_0^-(\gamma_x^-(t))-\wh u_0^-(\gamma_x^-(s))=\int_s^t\wh{L}(\gamma_x^-(\tau),\dot{\gamma}_x^-(\tau))+c(H)\mbox{d}\tau
\]
for all $s\leq t\leq 0$. }After all, $\wh u_0^-$ has to be a backward $0-$weak KAM solution of \eqref{eq:hj-sym0}. Consequently, $u_0^+=-\wh u_0^-$ has to be a forward $0-$weak KAM solution of (\ref{eq:hj}).\qed
\end{proof}



\vspace{20pt}

\noindent{\it Proof of Theorem \ref{thm:1}:} {It's a direct corollary from Lemma \ref{lem:sym-h},  \ref{lem:sym} and \ref{lem:for}.\qed

\vspace{40pt}

\appendix

\section{Aubry-Mather theory of nonsmooth convex Hamiltonians}\label{a1}

As is known, the continuous, \textr{superlinear}, convex $H(x,p)$ has a dual Lagrangian
\be\label{eq:leg}
L(x,v):=\max_{p\in T_x^*\T^n} \{\langle p,v\rangle-H(x,p)\},\quad (x,v)\in TM
\ee
which is also continuous, superlinear and convex in $v$. Consequently, for any $x,y\in M$ and $t>0$, the {\it action function}
\be\label{eq:act}
h^t(x,y):=\inf_{\substack{\gamma\in C^{ac}([0,t],M)\\\gamma(0)=x,\gamma(t)=y}}\int_0^t \big(L(\gamma,\dot\gamma)+c(H)\big) ds
\ee
always attains its infimum at an absolutely continuous minimizing curve $\gamma_{\min}:[0,t]\rightarrow M$
due to the {\sf Tonelli Theorem}.
In \cite{Ma2}, the {\sf Peierls barrier} function 
\be\label{eq:peierl}
h^\infty(x,y):=\liminf_{t\rightarrow+\infty}h^t(x,y)
\ee
is proved to be well-defined and continuous on $M\times M$.
\begin{defn}\cite{Ma2}
The {\sf projected Aubry set} is defined by 
\[
\cA:=\{x\in M: h^\infty(x,x)=0\}
\]
\end{defn}


Consider $TM$ (resp. $M$) as a measurable space and $\PP(TM)$ (resp. $\PP(M)$) by the set of all Borel probability measures on it.  A measure on $TM$ is denoted by $\wt\mu$, and we remove the tilde if we project it to $M$. We say that a sequence
$\{\wt\mu_n \}_n$ of probability measures weakly converges to a probability 
measure $\wt\mu$ if
\[
\lim_{n\rightarrow+\infty}\int_{TM}f(x,v)d\wt\mu_n(x,v)=\int_{TM} f(x,v)d\wt\mu(x,v)
\]
for any $f\in C_c(TM,\R)$. Accordingly, the deduced probability measure $\mu_n$ weakly converges to $\mu$, i.e.
\be\label{eq:pro-mes}
\lim_{n\rightarrow+\infty}\int_M f(x)d\mu_n(x)&:=&\lim_{n\rightarrow+\infty}\int_{TM} f(x)d\pi^*\wt\mu_n(x)\\
&=&\lim_{n\rightarrow+\infty}\int_{TM} f\circ \pi(x,v) d\wt\mu_n(x,v)\nonumber\\
&=&\int_{TM} f\circ\pi(x,v)d\wt\mu(x,v)\\
&=&\int_Mf(x)d\pi^*\wt\mu(x)=:\int_M f(x)d\mu(x)
\ee
for any $f\in C(M,\R)$.

\begin{defn}
A probability measure $\wt\mu$ on $TM$ is {\sf closed} if it satisfies:
\begin{itemize}
\item $\int_{TM}|v|d\wt\mu(x,v)<+\infty$;
\item $\int_{TM}\langle \nabla\phi(x),v\rangle d\wt\mu(x,v)=0$ for every $\phi\in C^1(M,\R)$.
\end{itemize}
\end{defn}
Let's denote by $\PP_c(TM)$ the set of all closed measures on $TM$, then the following conclusion is proved in \cite{DFIZ}:
\begin{thm}\label{thm:mane}
$\min_{\wt\mu\in\PP_c(TM)}\int_{TM}L(x,v)d\widetilde{\mu}(x,v)=-c(H)$. Moreover, the minimizer is called a {\sf Mather measure} and we denote by $\wt{\mathfrak M}$ the set of them. Similarly, we can project $\wt{\mathfrak M}$ to $\mathfrak M\subset\PP(M)$ w.r.t. $\pi:TM\rightarrow M$, which contains all the {\sf projected Mather measures}.
\end{thm}
\begin{defn}\cite{Ma}\label{defn:mat}
The {\sf Mather set} is defined by 
\[
\wt\cM:=\overline{\bigcup_{\tilde{\mu}\in\wt{\mathfrak M}}supp(\tilde{\mu})}\subset TM
\]
and the  {\sf projected Mather set} $\cM:=\pi\wt\cM$ is accordingly defined. 
\end{defn} 

\begin{defn}[subsolution]\label{defn:sub-sol}
 A function $u:M\rightarrow\R$ is called a {\sf viscosity subsolution}, or {\sf subsolution} for short of 
 \[
 H(x,d_x u)= c,\quad x\in M
 \]
 (denoted by $u\prec L+c$), if $u(y)-u(x)\leq h^t(x,y)+(c-c(H))t$ for all $(x,y)\in M\times M$ and $t\geq 0$. 
\end{defn}
\begin{defn}\label{defn:w-kam}
A function $u:M\rightarrow \R$ is called a {\sf backward (resp. forward) $0-$weak KAM solution} of (\ref{eq:hj}) if it satisfies:
\begin{itemize}
\item $u\prec L+c(H)$, i.e. for any two points $(x,y)\in M\times M$ and any absolutely continuous curve $\gamma:[a,b]\rightarrow M$ connecting them, we have 
\[
u(y)-u(x)\leq \int_a^b \big( L(\gamma,\dot\gamma)+c(H)\big) dt.
\]
\item for any $x\in M$ there exists a curve $\gamma_x^-:(-\infty,0]\rightarrow M$ (resp. $\gamma_x^+:[0,+\infty)\rightarrow M$) ending with (resp. starting from) $x$, such that for any $s<t\leq 0$ (resp. $0\leq s<t$), 
\[
u(\gamma_x^-(t))-u(\gamma_x^-(s))=\int_s^t \big( L(\gamma_x^-,\dot\gamma_x^-)+c(H) \big) dt
\]
\[
\bigg(resp.\quad u(\gamma_x^+(t))-u(\gamma_x^+(s))=\int_s^t \big( L(\gamma_x^+,\dot\gamma_x^+)+c(H)\big) dt\bigg).
\]
\end{itemize} 
\end{defn}


\vspace{40pt}


\begin{thebibliography}{}
\bibitem{CEL} Crandall, M. G., Evans, L. C., \& Lions, P.-L. {\it Some properties of viscosity solutions of Hamilton-Jacobi equations}. Trans. Amer. Math. Soc. 282 (1984), no. 2 487-502.
\bibitem{CI} Contreras G. \& Iturriaga R., {\it Global minimizers of autonomous Lagrangians}, 22$^\circ$ Col\'oquio Brasileiro de Matem\'atica. In: 22nd Brazilian Mathematics Colloquium, Instituto de Matem\'atica Pura e Aplicada (IMPA), Rio de Janeiro (1999).
\bibitem{CP} Gonzalo Contreras \& Gabriel P. Paternain, {\it Connecting orbits between static classes for generic Lagrangian systems}, Topology 41 (2002), no. 4, 645-666.
\bibitem{CIPP} Contreras G, Iturriaga R, Paternain G.P. \& Paternain M. {\it Lagrangian graphs, minimizing measures and Ma\~n\'e's critical values}. Geometric and Functional Analysis, 8(1998), no. 5, 788-809.
\bibitem{DFIZ} Davini A, Fathi A, Itturiaga R. \& Zavidovique M., {\it Convergence of the solutions of the discounted Hamilton-Jacobi equation}, Invent. Math. 206 (2016), no. 1, 29-55.
\bibitem{DS} Davini A. \& Siconolfi A., {\it A generalized dynamical approach to the large time behavior of solutions of Hamilton-Jacobi equations}, SIAM J. Math. Anal., Vol. 38, No.2, pp478-502, 2006.
\bibitem{F} Fathi A., {\it weak KAM theorems in Lagrangian dynamics}, version 10 unpublished, 2008.
\bibitem{LPV} Lions P.-L., Papanicolaou G. \& Varadhan S. {\it Homogenization of Hamilton-Jacobi equation}, unpublished preprint, (1987)
\bibitem{Mn} Ma\~n\'e R., {\it Generic properties and problems of minimizing measures of Lagrangian systems}. Nonlinearity 9 (1996), no. 2, 273-310.
\bibitem{Ma} Mather J., {\it Action minimizing invariant measures for positive definite Lagrangian systems}. Math. Z. 207 (1991), no. 2, 169-207.
\bibitem{Ma2} Mather J., {\it Variational construction of connecting orbits}. Ann. Inst. Fourier (Grenoble) 43 (1993), no. 5, 1349-1386.
\end{thebibliography}
\end{document}